\documentclass[a4paper,english]{amsart}
\usepackage[utf8]{inputenc}
\usepackage[T1]{fontenc}
\usepackage[american]{babel}
\usepackage{amsmath,amssymb,amsthm}
\usepackage{comment}
\usepackage{braket}
\usepackage{enumitem}
\usepackage{mathrsfs}

\newcommand{\F}{\mathbb{F}}
\newcommand{\cF}{\mathcal{F}}
\renewcommand{\P}{\mathcal P}
\newcommand{\T}{\mathbb{T}}
\newcommand{\com}{\mathbb{C}}
\newcommand{\real}{\mathbb{R}}
\newcommand{\A}{\mathcal{A}}
\newcommand{\cS}{\mathcal{S}}
\newcommand{\id}{\text{id}}
\newcommand{\cN}{\mathcal{N}}
\newcommand{\cM}{\mathcal{M}}
\newcommand{\N}{\mathbb{N}}

\theoremstyle{plain}
\newtheorem{thm}{Theorem}[section]
\theoremstyle{plain}
\newtheorem{prop}[thm]{Proposition}
\theoremstyle{remark}
\newtheorem{rem}[thm]{Remark}

\theoremstyle{plain}

\theoremstyle{plain}

\theoremstyle{plain}
\newtheorem{lem}[thm]{Lemma}
\theoremstyle{definition}

\theoremstyle{definition}

\title[Holomorphic heat-smoothing]{Heat-smoothing for holomorphic subalgebras of free group von Neumann algebras}

	\author{Haonan Zhang}

\address{Institute of Science and Technology Austria (IST Austria),
	Am Campus 1, 3400 Klosterneuburg, Austria}

\email{haonan.zhang@ist.ac.at}

\begin{document}
	\maketitle
	\begin{abstract}
		The heat semigroup on discrete hypercubes is well-known to be contractive over $L_p$-spaces for $1<p<\infty$. A question of Mendel and Naor \cite{MN14} concerns a stronger contraction property in the tail spaces, which is known as the heat-smoothing conjecture.  Eskenazis and Ivanisvili \cite{EI20} considered a Gaussian analog of this conjecture and resolved some special cases. In particular, they proved that heat-smoothing type conjecture holds for holomorphic functions in the Gaussian spaces with sharp constants. In this paper, we prove analogous sharp inequalities for holomorphic subalgebras of free group von Neumann algebras. Similar results also hold for $q$-Gaussian algebras and quantum tori. In the case of free group von Neumann algebras, the weaker formulation of heat-smoothing is proved with optimal order.
	\end{abstract}

\section{Introduction}
For $n\ge 1$, let $\{\pm 1\}^n=\{-1,1\}^n$ be the $n$-dimensional discrete hypercube. Any function $f:\{\pm 1\}^n\to \real$ admits the Fourier--Walsh expansion 
$$f=\sum_{S\subset [n]}\widehat{f}(S)\chi_S,$$
where for each $S\subset [n]:=\{1,\dots, n\}$, $\widehat{f}(S)\in \real$ and 
$$\chi_S(x):=\prod_{j\in S}x_j, \qquad x=(x_1,\dots, x_n)\in \{\pm 1\}^n.$$ 
Then the heat semigroup $T_t=e^{-tL},t\ge 0$ is defined as 
$$T_tf:=\sum_{S\subset [n]} e^{-t|S|}\widehat{f}(S)\chi_S,$$
or equivalently, $L(\chi_S)= |S|\chi_S$. Here $|S|$ denotes the volume of $S$. 

Let $1\le p\le \infty$ and denote by $L_p(d\mu_n)$ the $L_p$-space with respect to the uniform probability measure $\mu_n$ on $\{\pm 1\}^n$. Each $T_t$ is unital positive and $\mu_n(T_t f)=\mu_n(f)$ for all $f$, thus it extends to a contraction over $L_p(d\mu_n)$ for all $1\le p\le \infty$. The so-called \emph{heat-smoothing conjecture} states that for any $1<p<\infty$, there exists $c_p>0$ depending only on $p$ such that for any $n\ge 1$, $1\le d\le n$ and any $f:\{\pm 1\}^n\to \real$ with
\begin{equation}\label{eq:tail condition}
\widehat{f}(S)=0 \qquad\text{ whenever }\qquad|S|< d,
\end{equation}
we have
\begin{equation*}
\|T_t f\|_{L_p(d\mu_n)}\le e^{-c_p dt }\|f\|_{L_p(d\mu_n)},\qquad t\ge 0.
\end{equation*}
A weaker formulation is that, for all $f$ satisfying \eqref{eq:tail condition} we have
\begin{equation*}
\|L f\|_{L_p(d\mu_n)}\ge  c_p d\|f\|_{L_p(d\mu_n)}.
\end{equation*}
This question was first asked by Mendel and Naor \cite[Remark 5.5]{MN14} in the vector-valued setting, but the scalar case still remains open. Some partial results were known, and we refer to \cite{MN14,HMO17,EI20} for more information.

In \cite{EI20} Eskenazis and Ivanisvili considered a Gaussian analog of this conjecture by replacing the heat semigroup $T_t=e^{-tL}$ on the discrete hypercube $\left(\{\pm 1\}^n,d\mu_n\right)$ with the Ornstein--Uhlenbeck semigroup $T_t=e^{-tN}$ on the Gaussian space $(\real^n,d\gamma_n)$, where $\gamma_n$ is the standard Gaussian on $\real^n$ and $N=-\Delta+x\cdot \nabla$. The Ornstein--Uhlenbeck semigroup $T_t=e^{-tN}$ is unital positive and preserves the standard Gaussian measure, thus it acts a contraction over $L_p(d\gamma_n)=L_p(\real^n,d\gamma_n),1\le p\le \infty$. The spectrum of $N$ is $\N=\{0,1,\cdots\}$ and the eigenvectors of $N$ are Hermite polynomials. This allows to define for each $d\ge 1$ the {\em $d$-th tail spaces $\P^{\ge d}$}, which is spanned by Hermite polynomials of degree $\ge d$, satisfying a condition similar to \eqref{eq:tail condition}. Then it is natural to ask the Gaussian analog of heat-smoothing conjecture. By central limit theorem \cite{EI20}, the discrete hypercube heat-smoothing implies its Gaussian counterpart. In the Gaussian setting, Eskenazis and Ivanisvili \cite{EI20} resolved the conjecture in some special cases. In particular, they obtained sharp Gaussian heat-smoothing when restricted to holomorphic functions. More precisely, they showed that for any $n\ge 1, d\ge 1$ and $1\le p<\infty$, we have \cite[Theorem 3 and Lemma 9]{EI20}
\begin{equation}\label{ineq:heat-smoothing classical gaussian}
\left\| \sum_{|\alpha|\ge d}e^{-t|\alpha|}c_{\alpha}z^{\alpha}\right\|_{L_p(\gamma_{2n})}
\le e^{-td}\left\| \sum_{|\alpha|\ge d}c_{\alpha}z^{\alpha}\right\|_{L_p(\gamma_{2n})},
\end{equation}
and 
\begin{equation}\label{ineq:heat-smoothing classical gaussian bis}
\left\| \sum_{|\alpha|\ge d}|\alpha|c_{\alpha}z^{\alpha}\right\|_{L_p(\gamma_{2n})}
\ge d\left\| \sum_{|\alpha|\ge d}c_{\alpha}z^{\alpha}\right\|_{L_p(\gamma_{2n})}.
\end{equation}
Here we identify $\com^n$ with $\real^{2n}$ via $z_j=x_j+iy_j$ for $z=(z_1,\dots, z_n)\in \com^n$ and $(x_1,y_1,\dots, x_n, y_n)\in\real^{2n}$, and we have used the multi-index convention $z^{\alpha}=z_1^{\alpha_1}\cdots z^{\alpha_n}_{n}$ for $\alpha=(\alpha_1,\dots, \alpha_n)\in \N^{n}$. The polynomials $z^\alpha,\alpha\in\N^n$ are identified with the Hermite polynomials via the Segal--Bargmann transform. Then the condition $|\alpha|:=\alpha_1+\cdots +\alpha_n\ge d$ is equivalent to saying that the holomorphic polynomials (finite sum) $f(z)=\sum_{|\alpha|\ge d}c_\alpha z^\alpha$ belong to the $d$-th tail space $\P^{\ge d}$. A similar result \cite[Lemma 10]{EI20} holds for holomorphic $f\in \P^{\le d}$ i.e.  $f$ is of degree at most $d$. Eskenazis and Ivanisvili \cite[Theorem 5]{EI20} also obtained some moment comparison results, and all these inequalities are sharp. 

The proofs of \eqref{ineq:heat-smoothing classical gaussian} and \eqref{ineq:heat-smoothing classical gaussian bis} are very simple, and they also work in the vector-valued setting for {\em all} Banach spaces without any modifications. This highlights a big difference between the holomorphic and non-holomorphic settings. 

\medskip

As already seen in the Gaussian setting, the heat-smoothing conjecture can be formulated in a much more general context. In this paper, we shall prove several analogous sharp holomorphic heat-smoothing results in the noncommutative setting, following \cite{EI20}. Our main results hold for more general examples; see Section \ref{subsec:ending remark} at the end. Here we choose the free group von Neumann algebras, $q$-Gaussian algebras and quantum tori, three representative examples in noncommutative analysis, quantum probability and noncommutative geometry, to present the results and ideas. In the introduction and following, we will formulate and prove the main results for the free group von Neumann algebras in more detail. For the $q$-Gaussian algebras and the quantum tori, we will only briefly summarize the main results and their proof ingredients as they are very similar to the free group von Neumann algebras setting.

\medskip

{\bf Main results for free group von Neumann algebras.} For $n\ge 1$, we denote $\F_n$ the free group on $n$ generators and $\widehat{\F}_n$ the von Neumann subalgebra of $B(\ell_2(\F_n))$ generated by $\lambda(\F_n)$. Here $\lambda$ is the left regular representation. Let $\tau$ be the canonical tracial state over $\widehat{\F}_n$ and $L_p(\widehat{\F}_n)$ the associated noncommutative $L_p$-space. With the natural word length function $|\cdot|$ on $\F_n$, we may define the degree of a polynomial $\sum x_g \lambda_g,x_g\in \com$ in $\widehat{\F}_n$. This 
leads to, for each $d\ge 1$, the  space $\P^{\le d}$ of polynomials of degree $\le d$ and the $d$-th tail space $\P^{\ge d}$ that is spanned by the polynomials of degree $\ge d$. Our semigroup is now the Poisson-like semigroup $T_t=e^{-tL}$ of multipliers induced by the word length $|\cdot|$, i.e. $L(\lambda_g)=|g|\lambda_g$. Each $T_t$ is unital completely positive trace preserving, thus contractive on $L_p(\widehat{\F}_n)$ for all $1\le p\le \infty$. Then naturally one may ask the free heat-smoothing conjecture. Our main results for free groups are analogous to the above-mentioned sharp holomorphic inequalities by Eskenazis and Ivanisvili. Holomorphic polynomials in $\widehat{\F}_n$ are linear combinations of $\lambda_g,g\in \F_n^+$ where $\F_n^+$ is the semigroup generated by the generators (excluding the inverses) of $\F_n$. See Section \ref{subsec:free} for detailed definitions and conventions. 

\begin{thm}\label{thm:heat smoothing free}
	For any $d\ge 1,n\ge 1$, $1\le p\le \infty$ and holomorphic polynomial $x$ in $\widehat{\F}_n$, 
	\begin{enumerate}[label=(\roman*)]
		\item if $x$ belongs to the $d$-th tail space $\P^{\ge d}$, then
		\begin{equation}\label{ineq:heat smoothing tail}
		\|P_t(x)\|_p\le e^{-td}\|x\|_p,  \qquad t\ge 0,
		\end{equation}
		and 
		\begin{equation}\label{ineq:spectral gap tail}
		\|L (x)\|_p\ge d\|x\|_p;
		\end{equation}
		
		\item if $x$ belongs to the low-degree space $\P^{\le d}$, then 
		\begin{equation}\label{ineq:spectral gap low degree}
		\|L (x)\|_p\le d\|x\|_p,
		\end{equation}
		and 
		\begin{equation}\label{ineq:heat smoothing low degree}
		\|P_t(x)\|_p\ge e^{-td}\|x\|_p, \qquad t\ge 0.
		\end{equation}
	\end{enumerate}
\end{thm}

\begin{rem}
	Clearly the inequalities in Theorem \ref{thm:heat smoothing free} are sharp, as can be seen by choosing $d$-homogeneous holomorphic polynomials, e.g. $x=\lambda(g_1^d)$.
\end{rem}

The heat-smoothing is related to the hypercontractivity. On one hand, the latter implies the former for large time $t$ when $p\ge 2$ and the reference measure is finite. This is a standard argument \cite[Lemma 5.4]{MN14}, and we will repeat it in Section \ref{subsec:free} for reader's convenience. 

\begin{prop}\label{prop:heat smoothing via hypercontractivity}
	For any $2\le p<\infty$, there exists $c_p>0$ such that for any $d\ge 0,x\in \P^{\ge d}$ and $t\ge 0$
	\begin{equation}\label{ineq:hs via hc-semigroup}
	\|P_t(x)\|_p\le e^{-c_pd\min\{t,t^2\}}\|x\|_p.
	\end{equation}
	Hence for some $c_p'>0$ we have
	\begin{equation}\label{ineq:hs via hc-generator}
	\|L(x)\|_p\ge c_p'\sqrt{d}\|x\|_{p}.
	\end{equation}
\end{prop}

 On the other hand, combining hypercontractivity and heat-smoothing, one obtains the following moment comparison estimates:
\begin{prop}\label{prop:moment comparison}
	For holomorphic $x\in \P^{\le d}\subset \widehat{\F}_n$, we have 
	\begin{equation*}
	\|x\|_q\le c(p,q,d)\|x\|_p,
	\end{equation*}
	for
	\begin{enumerate}[label=(\roman*)]
		\item $1<p\le q<\infty$ and $c(p,q,d)=(\frac{q-1}{p-1})^{d}$;
		\item $p=2$, $q\ge q_0\approx 3.82$, and $c(p,q,d)=(q-1)^{d/2}$.
	\end{enumerate}
\end{prop}

In fact, \eqref{ineq:hs via hc-generator} can be improved to $O(d)$ which is the right order, and it is valid for all $1\le p\le \infty$. This is better than the estimates in the classical case and is special for free group von Neumann algebras (compared with the other examples that we are going to discuss). We are indebted to the referee for the following 

\begin{thm}\label{thm:free sharp order}
	For any $1\le p\le \infty$, $d\ge 1$, $x\in \P^{\ge d}$ and $t\ge 0$, we have 
	\begin{equation}\label{ineq:free hs sharp semigroup}
		\|P_t(x)\|_p\le (d+3)e^{-td}\|x\|_p,
	\end{equation}
and 
	\begin{equation}\label{ineq:free hs sharp generator}
		\|L(x)\|_p\ge \frac{d}{4}\|x\|_{p}.
	\end{equation}
\end{thm}
See the proof of Theorem \ref{thm:free sharp order} in Sect. \ref{subsec:free} for better constants in \eqref{ineq:free hs sharp semigroup} and \eqref{ineq:free hs sharp generator}.

\medskip

{\bf Main results for $q$-Gaussian algebras and quantum tori.}
Our main results for $q$-Gaussian algebras are similar sharp heat-smoothing type inequalities for holomorphic polynomials:

\begin{thm}\label{thm:heat smoothing q-gaussian}
	For any $q\in (-1,1)$, any real Hilbert space $H$ and any $1\le p\le \infty$, we have
	the same sharp holomorphic inequalities as in Theorem \ref{thm:heat smoothing free} for the $q$-Ornstein--Uhlenbeck semigroup $P_t^q=e^{-tN_q}$ over the $q$-Gaussian algebra $\Gamma_q(H\otimes H)$ with adapted conventions introduced in Section \ref{subsec:q-gaussian}.
\end{thm}

See Section \ref{subsec:q-gaussian} for the definitions. For $q$-Ornstein--Uhlenbeck semigroup $P_t^q$, we have a better hypercontractivity result, so the moment comparison estimates are slightly better than Proposition \ref{prop:moment comparison}:

\begin{prop}\label{prop:moment comparison q}
	For $q\in (-1,1)$ and any holomorphic $x\in \P^{\le d}$, we have 
	\begin{equation*}
	\|x\|_r\le c(p,r,d)\|x\|_p,
	\end{equation*}
	for
	\begin{enumerate}[label=(\roman*)]
		\item $1<p\le r<\infty$ and $c(p,r,d)=(\frac{r-1}{p-1})^{\frac{d}{2}}$;
		\item $p=2$, $r$ an even integer, and $c(p,r,d)=(\frac{r}{2})^{d/2}$.
	\end{enumerate}
\end{prop}

\medskip

For quantum tori, the sharp holomorphic heat-smoothing is still valid, but we do not have moment comparison due to the lack of hypercontractivity estimates. We refer to Section \ref{subsec:quantum tori} for the definitions. 

\begin{thm}\label{thm:heat smoothing quantum tori}
	For any $n\ge 1$, any real skewed symmetric $n$-by-$n$ matrix $\theta$ and any $1\le p\le \infty$, we have
	the same sharp holomorphic inequalities as in Theorem \ref{thm:heat smoothing free} for the Poisson semigroup $P_t^{\theta}=e^{-tL_\theta}$ over the quantum tori $\A_{\theta}$ with adapted conventions introduced in Section \ref{subsec:quantum tori}.
\end{thm}

We conclude the introduction with the following two remarks. When $d=1$, then (weak formulation of) heat-smoothing reduces to the {\em $L_p$-spectral gap inequality}. This was proved both in the commutative \cite{HMO17} and noncommutative setting \cite{CAPR2018spectral}. 

That many inequalities take stronger forms when restricting to holomorphic functions have already been studied in noncommutative setting. For example, Kemp \cite{Kemp05} proved a stronger version of hypercontractivity for $q$-Gaussian algebras (see also Section \ref{sect:preliminary} below) as a noncommutative analog of Janson's strong hypercontractivity theorem \cite{Janson83}. Kemp and Speicher \cite{KS07} observed that the celebrated Haagerup inequality, which plays a crucial role in many different areas, can be improved when restricting to the holomorphic ($\mathscr{R}$-diagonal) elements. 

\section{Preliminary}\label{sect:preliminary}
\subsection{Noncommutative $L_p$-spaces}
As it will be used in different examples, we briefly recall the noncommutative $L_p$-space here for convenience. Let $\cM$ be a finite von Neumann algebra equipped with a normal faithful tracial state $\tau$. For any $1\le p <\infty$, we define for any $x\in \cM$
\begin{equation*}
\|x\|_p^p:=\tau(x^\ast x)^{p/2}.
\end{equation*}
Then $\|\cdot\|_p$ is a norm and the noncommutative $L_p$-space associated with $(\cM,\tau)$ is defined as the completion of $(\cM,\|\cdot\|_p)$, denoted $L_p(\cM,\tau)$ or $L_p(\cM)$ for short. For $p=\infty$, we define $L_{\infty}(\cM)$ as $\cM$ equipped with the operator norm $\|\cdot\|_\infty:=\|\cdot\|$. Noncommutative $L_p$-spaces share many properties with the classical ones, such as H\"older's inequality: For $ 1\le p,q,r\le \infty$ and $\theta\in [0,1]$ such that $1/r=1/p+1/q$,
\begin{equation}\label{ineq:holder}
\|xy\|_r\le \|x\|_p\|y\|_q,\qquad x\in L_p(\cM),y\in L_q(\cM).
\end{equation}
We refer to \cite{PisierXu2003LP} for more detail. In the following, we will use $\|\cdot\|_p$ among different examples of $(\cM,\tau)$ as there is no ambiguity.

\subsection{Free group von Neumann algebras}\label{subsec:free}
Let $\F_n$ be the free group on $n$ generators $g_j,1\le j\le n$ with $e$ the unit element. Let $\widehat{\F}_n$ denote the group von Neumann algebra of $\F_n$, that is, the von Neumann algebra acting on $\ell_2(\F_n)$ generated by $\lambda_g,g\in \F_n$, where $\lambda$  is the left regular representation, i.e. $\lambda_gf(h):=f(g^{-1}h),f\in \ell_2(\F_n)$. We use $\tau$ to denote the canonical tracial state on $\widehat{\F}_n$, that is, $\tau(x)=\langle x\delta_e,\delta_e\rangle$ with $\langle\cdot,\cdot\rangle$ being the usual inner product on $\ell_2(\F_n)$ and $\delta_e$ being the delta function at $e$. All the polynomials $\sum_{g}x_g\lambda_g,x_g\in\com$, where $\sum_{g}$ always denotes the finite sum, form a norm-dense $\ast$-subalgebra of $\widehat{\F}_n$ that is enough for our use. In particular, 
$$\tau\left(\sum_{g}x_g\lambda_g\right)=x_e.$$

%We denote by $L_p(\widehat{\F}_n)$ the noncommutative $L_p$-space associated with $(\widehat{\F},\tau)$. The $L_p$-norm is given by $\|x\|_p^p:=\tau (x^\ast x)^{p/2}$.

Any element $g\in \F_n$ has a word representation, i.e. it can be uniquely represented as a finite product of $g_i,g_j^{-1},1\le i,j\le n$, and it does not contain $g_k g_k^{-1}$ or $g_k^{-1}g_k$. This gives the word length function $|\cdot|:\F_n\to \N$, that is, $|g|$ denotes the number of $g_i,g_j^{-1},1\le i,j\le n$ in the word representation of $g$. For example, $|e|=0$ and $|g_1g_2^{-1}g_3^2|=4$. A classical result of Haagerup \cite[Lemma 1.2]{Haagerup79} states that for any $t>0$, $e^{-t|\cdot|}$ is positive semi-definite on $\F_n$. So the semigroup of linear operators $P_t=e^{-tL},t\ge 0$ over $\widehat{\F}_n$ given by
\begin{equation}\label{eq:defn of semigroup-free}
P_t\left(\sum_g x_g \lambda_g\right):= \sum_g e^{-t|g|} x_g \lambda_g,
\end{equation}
is unital completely positive trace preserving. It extends to a contraction over $L_p(\widehat{\F}_n)$ for all $1\le p\le \infty$. Note that the generator $L$ acts as
\begin{equation}\label{eq:defn of generator-free}
L\left(\sum_g x_g \lambda_g\right)= \sum_g |g|x_g \lambda_g.
\end{equation}

For any $d\ge 0$, we denote $\P^{\le d}$ (resp. $\P^{\ge d}$) the $\com$-linear span of $\{\lambda_g,|g|\le d\}$ (resp. $\{\lambda_g,|g|\ge d\}$). Denote $\F_n^+$ the semigroup generated by the generators $\{g_1,\dots,g_n\}$:
$$\F_n^+:=\{g\in \F_n:g=g_{j_1}^{k_1}\cdots g_{j_m}^{k_m}, \quad k_l> 0,1\le j_l\le n, j_l\neq j_{l+1},m \ge 0\}.$$
 A polynomial $x\in \widehat{\F}_n$ is \emph{holomorphic} if it is of the form $x=\sum_{g\in \F_n^+} x_g\lambda_g$.
All the holomorphic polynomials form a subalgebra of $\widehat{\F}_n$.

It is well-known that for any $z\in\T$, there exists a $\ast$-automorphism $\pi_z$ of $\widehat{\F}_n$ such that $\pi_z(\lambda(g_{j}))=z\lambda(g_{j}),1\le j\le n$. Clearly, $\pi_z$ is trace preserving. So it extends to an isometry on $L_p(\widehat{\F}_n)$: For any $1\le p\le \infty$ and $x\in L_p(\widehat{\F}_n)$,
\begin{equation}\label{eq:rotation free}
\left\|\pi_z(x) \right\|_p=\left\|x\right\|_p.
\end{equation}
 Note that for any holomorphic $x=\sum_{g\in \F_n^+} x_g \lambda_g$, we have
\begin{equation}\label{eq:rotation over analytic-free}
\pi_z(x)=\sum_{g\in \F_n^+} z^{|g|}x_g \lambda_g.
\end{equation}

As already mentioned in the introduction, one may derive the heat-smoothing for large time and $p\ge 2$ using a standard argument \cite[Lemma 5.4]{MN14}, provided that the reference measure is finite. Although we do not have the optimal hypercontractivity for free group von Neumann algebras, the known results already yield the heat-smoothing with constants $e^{-c_pd\min\{t,t^2\}}$ for $p\ge 2$ that is of the same asymptotic behavior as in \cite{MN14}.
The proof is essentially the same as in \cite{MN14}. We provide it here for completeness.

\begin{proof}[Proof of Proposition \ref{prop:heat smoothing via hypercontractivity}]
	Compared with the proof of \cite[Lemma 5.4]{MN14}, the only difference is that we use the following hypercontractivity estimates for free group von Neumann algebras over twice the optimal time \cite[Theorem A]{JPPR15}: If $1< p\le q<\infty$, 
	\begin{equation*}
	\|P_t(x)\|_q\le \|x\|_p\qquad {\rm for}\qquad t\ge \log\frac{q-1}{p-1}.
	\end{equation*}
	For $p\ge 2$, take $q=1+e^t(p-1)\ge p$. Since $x\in \P^{\ge d}$ and using H\"older's inequality \eqref{ineq:holder} (recall that $\tau$ is a state):
	\begin{equation*}
	\|P_t (x)\|_2\le e^{-td}\|x\|_2\le e^{-td}\|x\|_p, \qquad t\ge 0.
	\end{equation*} 
	This, together with H\"older's inequality \eqref{ineq:holder} and above hypercontractivity estimates, yields
	\begin{equation*}
	\|P_t(x)\|_p\le \|P_t(x)\|_2^{\theta}\|P_t(x)\|_q^{1-\theta}\le e^{-td\theta}\|x\|_p,\qquad t\ge 0,
	\end{equation*}
	where $\theta\in [0,1]$ is such that 
	\begin{equation*}
	\frac{1}{p}=\frac{\theta}{2}+\frac{1-\theta}{q}.
	\end{equation*}
	So we have proved
	\begin{equation*}
	\|P_t(x)\|_p\le \exp\left(-\frac{2(p-1)dt(e^t-1)}{p(e^t(p-1)-1)}\right)\|x\|_p,\qquad 2\le p<\infty.
	\end{equation*}
	With this, we see that \eqref{ineq:hs via hc-semigroup} holds with $c_p=2/p$. In fact, this follows from
	\begin{equation*}
	\min\{t,t^2\}\le t\le \frac{t (e^t-1)}{e^t-(p-1)^{-1}}=\frac{(p-1)t(e^t-1)}{e^t(p-1)-1},
	\end{equation*}
	where we have used the fact that $p\ge 2$. As argued in \cite[Proof of Lemma 5.4]{MN14}, \eqref{ineq:hs via hc-generator} follows immediately from \eqref{ineq:hs via hc-semigroup}.
\end{proof}

\subsection{$q$-Gaussian algebras}\label{subsec:q-gaussian}
Our references for this part are \cite{BS91,BS94,BKS97,LP99,KS07}. Fix $q\in [-1,1]$. Let $H$ be a real Hilbert space and $H_\com$ the complexification with $\langle \cdot,\cdot\rangle$ the inner product that is linear in the second argument. The algebraic Fock space $\cF(H)$ is 
\begin{equation*}
\cF(H):=\oplus_{k=0}^{\infty}H_{\com}^{\otimes k},
\end{equation*} 
that is, $\cF(H)$ is the linear span of vectors of the form $\xi_1\otimes\cdots\otimes \xi_k\in H_{\com}^k,k\in\N$, where $H_{\com}^{0}:=\com \Omega$. Here $\Omega$ is some unit vector known as the {\em vacuum vector}. We define a Hermitian form $\langle \cdot,\cdot\rangle_q$ on $\cF(H)$ given by 
\begin{equation*}
\langle\Omega,\Omega\rangle_q:=1
\end{equation*}
and 
\begin{equation*}
\langle \xi_1\otimes \cdots \otimes \xi_j, \eta_1\otimes \cdots\otimes \eta_k\rangle_q:=\delta_{j,k}\sum_{\sigma\in \cS_k}q^{\iota(\sigma)}\langle \xi_1, \eta_{\sigma(1)}\rangle\cdots \langle \xi_k, \eta_{\sigma(k)}\rangle,
\end{equation*}
where $\cS_k$ denotes the permutation group on $k$ letters and $\iota(\sigma)$ denotes the number of inversions in $\sigma$ i.e. $\iota(\sigma):=\sharp \{(m,n):1\le m<n\le k\text{ and }\sigma(m)> \sigma(n)\}$. This way, the Hermitian form $\langle \cdot, \cdot\rangle_q$ is positive semi-definite for all $q\in [-1,1]$, and strictly positive if $q\in (-1,1)$. In the latter case, it defines an inner product on $\cF(H)$ for $q\in (-1,1)$ and we denote by $\cF_q(H)$ the Hilbert space obtained by taking completion of $\cF(H)$ with respect to $\langle \cdot, \cdot\rangle_q$. At the endpoint cases $q=\pm 1$, we need to take the quotient of $\cF(H)$ by the kernel of $\langle \cdot, \cdot\rangle_q$ first. Then after completion, we obtain a Hilbert space that is also denoted by $\cF_q(H)$ for $q=\pm 1$. To avoid technical issues regarding $q=\pm 1$, we assume that $-1<q<1$ in the following.

Now we define the $q$-Gaussian algebras $\Gamma_q(H)$ acting on the {\em Fock space} $\cF_q(H)$. For any vector $\xi \in H$, the {\em creation operator} $c_q(\xi)$ is a linear operator over $\cF_q(H)$ defined via
\begin{equation*}
c_q(\xi)\Omega:=\xi,
\end{equation*}
and
\begin{equation*}
c_q(\xi)\left(\xi_1\otimes \cdots\otimes \xi_k\right):=\xi\otimes \xi_1\otimes\cdots\otimes \xi_k.
\end{equation*}
The {\em annihilation operator} $c_q^\ast(\xi)$ is the adjoint of the creation operator $c_q(\xi)$ with respect to $\langle\cdot, \cdot\rangle_q$, or equivalently,
\begin{equation*}
c_q^\ast(\xi)\Omega=0,
\end{equation*}
and 
\begin{equation*}
c_q^\ast(\xi)\left(\xi_1\otimes \cdots\otimes  \xi_k\right)=\sum_{j=1}^{k}q^{j-1}\langle \xi,\xi_j\rangle \xi_1\otimes\cdots \otimes  \xi_{j-1}\otimes \xi_{j+1}\otimes \cdots \otimes \xi_k.
\end{equation*}
The creation and annihilation operators satisfy the {\em $q$-commutation relations}:
\begin{equation*}
c_q^\ast(\xi) c_q(\eta)-qc_q(\eta) c_q^\ast (\xi)=\langle \xi,\eta\rangle_q \id.
\end{equation*}
When $q\in (-1,1)$, the creation and the annihilation operators are bounded. For each $\xi\in H$, denote $X_q(\xi):=c_q(\xi)+c_q^\ast(\xi)$. Note that $\xi\mapsto X_q(\xi)$ is $\real$-linear. The {\em $q$-Gaussian algebra} $\Gamma_q(H)$ is the von Neumann subalgebra of $B(\cF_q(H))$ generated by these self-adjoint operators $X_q(\xi),\xi\in H$. We use $\tau_q$ to denote the vacuum state given by $\tau_q(A):=\langle A\Omega,\Omega \rangle_q$. It restricts to a faithful tracial state on $\Gamma_q(H)$. 

Any contraction $T$ over $H$ assigns to a bounded linear map $\Gamma_q(T)$, the {\em second quantization},  over $\Gamma_q(H)$ such that $\Gamma_q(T)X_q(\xi)=X_q(T\xi)$ for all $\xi\in H$. It is unital completely positive and trace preserving. If $T$ is isometric, then $\Gamma_q(T)$ is a faithful homomorphism. If $S$ is another contraction over $H$, then $\Gamma_q(ST)=\Gamma_q(S)\Gamma_q(T)$. The {\em $q$-Ornstein--Uhlenbeck} semigroup $P_t^q=e^{-tN_q}$ over $\Gamma_q(H)$ is the second quantization of $e^{-t}\id_{H}$, i.e.  $P_t^q=\Gamma_q(e^{-t}\id_{H})$, so it extends to a contraction over $L_p(\Gamma_q(H),\tau_q)$ for $1\le p\le \infty$. Moreover, it is hypercontractive: 

\begin{comment}
Its generator $N_q$ is defined as the second quantization of the {\em number operator $\cN_q$} over $\cF_q(H)$, which we now explain. The {\em number operator $\cN_q$} is densely defined and essentially self-adjoint over $\cF_q(H)$ such that 
\begin{equation*}
\cN_q \Omega=\Omega,
\end{equation*}
and 
\begin{equation*}
\cN_q \left(\xi_1\otimes \cdots \xi_k\right)=k \xi_1\otimes \cdots \xi_k.
\end{equation*}
The map $\phi:\Gamma_q(H)\to \cF_q(H),A\mapsto A\Omega$ extends to a unitary from $L_2(\Gamma_q(H),\tau_q)$ onto $\cF_q(H)$, that we denote by $\phi_q$. Then the generator of $q$-Ornstein--Uhlenbeck semigroup $P_t^q=e^{-tN_q}$ is given by $N_q=\phi_q^{-1}\cN_q \phi_q$. The $q$-Ornstein--Uhlenbeck semigroup $P_t^q=e^{-tN_q}$ is hypercontractive: 
\end{comment}

\begin{thm}\cite[Theorem 3]{Biane97hypercontractivity}\label{thm:hypercontracitivity for q}
	For $q\in (-1,1)$ and $1<p\le r<\infty$, we have
	\begin{equation*}
	\|P_t^q (x)\|_r\le \|x\|_p \qquad {\rm iff } \qquad t\ge \frac{1}{2}\log \frac{r-1}{p-1}.
	\end{equation*}
\end{thm}

Fix $H$ a real Hilbert space as above. We now consider the $q$-Gaussian algebra $\Gamma_q(H\oplus H)$, where $H\oplus H$ denotes the $\ell_2$-direct sum of $H$. For any $\xi\in H$, consider the vectors $(\xi,0),(0,\xi)\in H\oplus H$ and the associated $X_q$-operators $X_q(\xi,0),X_q(0,\xi)\in \Gamma_q(H\oplus H)$. Put 
\begin{equation*}
Z_q(\xi):=\frac{1}{\sqrt{2}}\left(X_q(\xi,0)+i X_q(0,\xi)\right).
\end{equation*}
This way, we defined a holomorphic variable as $z=(x+iy)/\sqrt{2}$ in the classical setting. Here the factor of $\sqrt{2}$ is just to make $Z_q(\xi)$ a unit vector in $L_2(\Gamma_q(H\oplus H),\tau_q)$, as did in \cite{Kemp05}. The algebra generated by these $Z_q(\xi),\xi\in H$ is the {\em $q$-holomorphic algebra} (actually Kemp considered the Banach algebra in \cite{Kemp05}). Suppose that $(e_j)$ is an orthonormal basis of $H$. Any element in the $q$-holomorphic algebra, which we shall call {\em holomorphic polynomial}, can be written as the linear combination of 
$$Z_q(e_{j_1})^{n_1}\cdots Z_q(e_{j_k})^{n_k},\qquad n_\ell\ge 0, j_\ell\neq j_{\ell+1}, 1\le \ell \le k-1, k\ge 1.$$
Each of the above basis is a eigenvector of $N_q$ with eigenvalue $n_1+\cdots +n_k$:
\begin{equation}\label{eq:generator over analytic-q gaussian}
N_q\left(Z_q(e_{j_1})^{n_1}\cdots Z_q(e_{j_k})^{n_k}\right)=(n_1+\cdots+ n_k)Z_q(e_{j_1})^{n_1}\cdots Z_q(e_{j_k})^{n_k},
\end{equation} 
and thus
\begin{equation}\label{eq:semigroup over analytic-q gaussian}
P_t^q\left(Z_q(e_{j_1})^{n_1}\cdots Z_q(e_{j_k})^{n_k}\right)=e^{-t(n_1+\cdots+ n_k)}Z_q(e_{j_1})^{n_1}\cdots Z_q(e_{j_k})^{n_k}.
\end{equation}
When restricting to the $q$-holomorphic algebra, Kemp \cite[Theorem 1.4]{Kemp05} proved some stronger hypercontractivity estimates as a noncommutative analog of Janson's strong hypercontractivity theorem \cite[Theorem 11]{Janson83}

\begin{thm}\label{thm:strong hypercontracitivity for q}\cite[Theorem 1.4]{Kemp05}
	Let $H$ be a real Hilbert space. For $q\in [-1,1)$ and $r$ an positive even integer, we have for all holomorphic polynomial $x\in \Gamma_q(H\oplus H)$:
	\begin{equation*}
	\|P_t^q (x)\|_r\le \|x\|_2 \qquad {\rm iff }\qquad t\ge \frac{1}{2}\log \frac{r}{2}.
	\end{equation*}
\end{thm}

Since we have (optimal) hypercontractivity for $q$-Ornstein--Uhlenbeck semigroup $P_t^q=e^{-tN_q}$ Theorem \ref{thm:hypercontracitivity for q}, we have a general heat-smoothing estimate analogous to Proposition \ref{prop:heat smoothing via hypercontractivity} which will not be stated here. 

We end this subsection on $q$-Gaussian algebras with the following crucial rotation-invariance property. Fix any $z=e^{2\pi i \theta}\in \T$ with $\theta\in [0,2\pi]$. Denote by $U_\theta$ the following unitary operator
\begin{equation*}
U_\theta:=\begin{pmatrix}
\cos\theta & \sin\theta\\
-\sin\theta & \cos\theta 
\end{pmatrix}
\otimes \id_{H}
\end{equation*}
over $H\oplus H$. The second quantization of these $U_\theta, \theta\in [0,2\pi]$ induces a semigroup (i.e. $\pi_z\circ\pi_{z'}=\pi_{zz'}$) $\pi_z:=\Gamma_q(U_\theta)$ of $\ast$-automorphisms of $\Gamma_q(H\oplus H)$. They are trace preserving and extend to isometries over $L_p(\Gamma_q(H\oplus H))$: For all $x\in L_p(\Gamma_q(H\oplus H))$ and all $1\le p\le \infty$,
\begin{equation}\label{eq:rotation invariance-q gaussian}
\|\pi_z(x)\|_p=\|x\|_p.
\end{equation}

%unitary over $\cF_q(H\oplus H)$ such that
%\begin{equation*}
%\cF_q(U_\theta)\Omega=\Omega,
%\end{equation*}
%and 
%\begin{equation*}
%\cF_q(U_\theta)\left(\xi_1\otimes \cdots \otimes \xi_k\right)=U_\theta\xi_1\otimes \cdots\otimes U_\theta\xi_k.
%\end{equation*}
%Consider the rotation operators
%\begin{equation*}
%\pi_z(\cdot):=\cF_q(U_\theta) \cdot \cF_q(U_\theta)^\ast,
%\end{equation*}
%which form a semigroup (i.e. $\pi_z\circ\pi_{z'}=\pi_{zz'}$) of $*$-automorphisms of $\Gamma_q(H\otimes H)$. 
Since $\pi_z \left(X_q(\eta)\right)=X_q(U_\theta \eta)$ for $\eta\in H\oplus H$, we have
for any $\xi \in H$ 
\begin{equation*}
\pi_z \left(X_q(\xi,0)\right)=\cos \theta X_q(\xi, 0)-\sin\theta X_q(0,\xi),
\end{equation*}
\begin{equation*}
\pi_z \left(X_q(0,\xi)\right)=\sin \theta X_q(\xi, 0)+\cos\theta X_q(0,\xi),
\end{equation*}
and thus 
\begin{equation*}
\pi_z \left(X_q(\xi,0)+iX_q(0,\xi)\right)=(\cos\theta+i\sin\theta)\left(X_q(\xi,0)+iX_q(0,\xi)\right).
\end{equation*}
Hence for any $z\in\T$ and any $\xi\in H$ we have
\begin{equation*}
\pi_z(Z_q(\xi))=zZ_q(\xi), \qquad 
\end{equation*}
and on the basis of $q$-holomorphic algebra
\begin{equation}\label{eq:rotation over analytic-q gaussian}
\pi_z\left(Z_q(e_{j_1})^{n_1}\cdots Z_q(e_{j_k})^{n_k}\right)=z^{n_1+\cdots n_k}Z_q(e_{j_1})^{n_1}\cdots Z_q(e_{j_k})^{n_k}.
\end{equation}
%To conclude, for any holomorphic $x$, we may write $x=\sum_{k=d}^{m}x_k$ for some $0\le d\le m$ and for some holomorphic $x_k$ such that $P_t^q(x_k)=e^{-tk}x_k, t\ge 0$. Moreover, we have for all $z\in \T$ and $1\le p\le \infty$ that
%\begin{equation}\label{eq:rotation q gaussian}
%\qquad \pi_z(x)=\sum_{k=d}^{m}z^k x_k \qquad {\rm and }\qquad 
%\left\| \pi_z(x)\right\|_p=\|x\|_p.
%\end{equation}

\subsection{Quantum tori}\label{subsec:quantum tori}
Let $n\ge 2$ and $\theta=(\theta_{jk})_{j,k=1}^{n}$ be a real skew-symmetric matrix. The {\em quantum tori} $\A_{\theta}$ is the universal $C^\ast$-algebra generated by $n$ unitary operators $U_j,1\le j\le n$ satisfying 
\begin{equation*}
U_j U_k=e^{2\pi i \theta_{jk}}U_k U_j, \qquad 1\le j,k\le n.
\end{equation*}
For a multi-index $\alpha=(\alpha_1,\dots, \alpha_n)\in \mathbb{Z}^n$, we denote
\begin{equation*}
U^\alpha:=U_{1}^{\alpha_1}\cdots U_{n}^{\alpha_n}.
\end{equation*}
Then the polynomials (finite sum here)
\begin{equation*}
\sum_{\alpha\in \mathbb{Z}^n}x_\alpha U^\alpha,\qquad x_\alpha\in\com,
\end{equation*}
form a dense $\ast$-subalgebra of $\A_{\theta}$. The functional 
\begin{equation*}
\tau_\theta\left(\sum_{\alpha\in \mathbb{Z}^n}x_\alpha U^\alpha\right):=x_0,
\end{equation*}
on polynomials extends to a faithful tracial state, which we still denote by $\tau_\theta$, on $\A_{\theta}$. Let $\T_{\theta}$ be the von Neumann algebra arising from the GNS construction with respect to $(\A_{\theta},\tau_\theta)$. When $\theta=0$, one recovers the classical $n$-dimensional tori, i.e. $\A_{0}\cong C(\T^n)$ and $\T_{0}\cong L_\infty(\T^n)$. We refer to \cite{CXY13harmonic} for more discussion. 
%For $1\le p\le \infty$, we denote by $L_p(\T_{\theta})$ the noncommutative $L_p$-space associated with $(\T_{\theta},\tau)$. The $L_p$-norm on polynomials is given by $\|x\|_p^p:=\tau (x^\ast x)^{p/2}$ for $p<\infty$ and $\|x\|_\infty:=\|x\|$ for $p=\infty$.

We shall work with the Poisson type semigroup $P_t^\theta=e^{-tL_\theta}$ given by 
\begin{equation}\label{eq:defn of semigroup-quantum tori}
P_t^\theta\left(\sum_{\alpha\in \mathbb{Z}^n}x_\alpha U^\alpha\right):=\sum_{\alpha\in \mathbb{Z}^n} e^{-t|\alpha|}x_\alpha U^\alpha,
\end{equation}
where $|\alpha|:=\sum_{j=1}^{n}|\alpha_j|$. Its generator $L_\theta$ acts as
\begin{equation}\label{eq:defn of generator-quantum tori}
L_\theta\left(\sum_{\alpha\in \mathbb{Z}^n}x_\alpha U^\alpha\right)=\sum_{\alpha\in \mathbb{Z}^n}|\alpha|x_\alpha U^\alpha.
\end{equation}
Each $P_t^\theta$ extends to a unital completely positive trace preserving map over $\T_{\theta}$, as can be proven using the transference in \cite{CXY13harmonic}, then a contraction over $L_p(\T_{\theta}), 1\le p\le \infty$.

For each $d\ge 1$, we denote by $\P^{\le d}$ (resp. $\P^{\ge d}$) the $\com$-linear span of all $U^\alpha$, $|\alpha|\le d$ (resp. $|\alpha|\ge d$). For any $z\in\T$, there exists a $\ast$-automorphism $\pi_z$ of  $\T_{\theta}$ such that 
\begin{equation*}
\pi_z(U_j)=z U_j, \qquad 1\le j\le n.
\end{equation*}
Clearly, it is also trace preserving. So it extends to an isometry over $L_p(\T_{\theta})$: For $1\le p\le \infty$ and $x\in L_p(\T_{\theta})$,
\begin{equation}\label{eq:rotation quantum tori}
\left\| \pi_z(x)\right\|_p=\left\| x\right\|_p.
\end{equation}
 A polynomial is \emph{holomorphic} if it is of the form $x=\sum_{\alpha\in \N^n} x_\alpha U^\alpha$, for which
\begin{equation}\label{eq:rotation-analytic-quantum tori}
\pi_z(x)=\sum_{\alpha\in \N^n} z^{\alpha_1+\cdots+\alpha_n}x_\alpha U^\alpha=\sum_{\alpha\in \N^n} z^{|\alpha|}x_\alpha U^\alpha.
\end{equation}
%Therefore for any holomorphic $x$ we may write $x=\sum_{k=d}^{m}x_k\in \A_{\theta}$ for some $0\le d\le m$ and holomorphic $x_k, d\le k\le m$ such that $P_t^\theta (x_k)=e^{-tk}x_k, t\ge 0$. Moreover, for all $z\in \T$ and $1\le p\le\infty$,

%We do not have moment estimates for quantum tori as no hypercontractivity estimates are known.

\section{Proof of main results}
This section is devoted to the proof of main results. We first give the proof of free group von Neumann algebra case in detail. Then we illustrate that similar arguments apply to the $q$-Gaussian algebras and quantum tori with minor modifications. 

\subsection{Proof of main results: free group von Neumann algebra case}
 We first prove Theorem \ref{thm:heat smoothing free}. As in the proof in the classical Gaussian setting \cite{EI20}, there are two main ingredients. One is the rotation invariance property \eqref{eq:rotation free}. The other one is the following well-known lemma used in \cite{EI20}. Let $C(\T)$ denote the set of continuous functions on the torus $\T$. For any complex Borel measure $\mu$ on $\T$ we use $\|\mu\|$ to denote its total variation norm. 

\begin{lem}\cite{EI20}\label{lem:key}
	Fix $d\ge 1$. Then 
	\begin{enumerate}[label=(\roman*)]
		\item for any $m\ge d$ and $t\ge 0$, there exists a complex Borel measure $\mu=\mu_{d,m,t}$ on $\T$ such that 
		\begin{equation*}
		\int_{\T} z^k d\mu(z)= e^{-tk},\quad d\le k\le m,  \quad {\rm and  } \qquad \|\mu\|\le e^{-td};
		\end{equation*}
		\item there exists a complex Borel measure $\nu=\nu_d$ on $\T$ such that 
		$$\int_{\T}z^kd\nu(z)=k, \quad 0\le k\le d,\qquad {\rm and } \qquad \|\nu\|\le d.$$
	\end{enumerate}
\end{lem}

\begin{proof}
	See the proof of \cite[Lemma 9 $\&$ 10]{EI20}.
\end{proof}

\begin{rem}\label{rem:kernel}
	We are grateful to the referee who pointed out to us that (i) in the above lemma also holds for $m=\infty$. To see this, just consider the measure $d\mu(z)=f(z)dz$ with $f(z)=e^{-td}z^{-d}\sum_{n\in \mathbb{Z}}e^{-t|n|}z^{n}$ a multiple of shifted Poisson kernel. Similarly, (ii) follows from by choosing $d\nu(z)=g(z)dz$ with $g(z)=z^{-d}\sum_{j=0}^{d-1}\sum_{i=-j}^{j}z^i$ a multiple of shifted Fej\'er kernel.
\end{rem}

\begin{proof}[Proof of Theorem \ref{thm:heat smoothing free}]
	 Let us prove (i) first. Suppose that $x=\sum_{|g|=d}^{m}x_g\lambda_g$ is holomorphic for some $m\ge d$. %As argued in the proof of \cite[Lemma 9]{EI20}, there exists a complex Borel measure $\mu$ on $\T$ such that 
	%\begin{equation*}
	%\int_{\T} z^k d\mu(z)= e^{-tk}, d\le k\le m,  \text{  and  }\|\mu\|\le e^{-td}.
	%\end{equation*}
	Fix $t\ge 0$, and let $\mu=\mu_{d,m,t}$ be the measure in Lemma \ref{lem:key} (i). Then by \eqref{eq:defn of semigroup-free} and \eqref{eq:rotation over analytic-free}
	\begin{equation}\label{eq:average-semigroup-free}
	P_t(x)=\sum_{|g|=d}^{m} e^{-t|g|}x_g\lambda_g
	=\int_{\T}\sum_{|g|=d}^{m} z^{|g|}x_g\lambda_g d\mu(z)
	=\int_{\T} \pi_z(x) d\mu(z).
	\end{equation}
	This, together with the triangular inequality, \eqref{eq:rotation free} and the definition of $\mu$, yields
	\begin{align*}
	\|P_t(x)\|_p\le \int_{\T} \left\|\pi_z(x)\right\|_pd|\mu|(z)
	=\int_{\T} \left\|x\right\|_pd|\mu|(z)
	\le e^{-td}\|x\|_p.
	\end{align*}
	This finishes the proof of \eqref{ineq:heat smoothing tail}.  \eqref{ineq:spectral gap tail} is an immediate consequence of \eqref{ineq:heat smoothing tail} (note that $L(x)$ is also holomorphic and belongs to $\P^{\ge d}$):
	\begin{equation*}
	\|x\|_p=
	\left\|\int_{0}^{\infty}P_t L(x) dt\right\|_p\le \int_{0}^{\infty}\|P_t L(x)\|_p dt
	\le \int_{0}^{\infty} e^{-td} dt \|L(x) \|_p= d^{-1}\|L(x)\|_p.
	\end{equation*}
	
	Now we prove (ii). Recall that $x=\sum_{|g|\le d}x_g\lambda_g$ is holomorphic. 
	Let $\nu=\nu_d$ be the measure in Lemma \ref{lem:key} (ii). Then by \eqref{eq:defn of generator-free} and \eqref{eq:rotation over analytic-free},
	\begin{equation}\label{eq:average-generator-free}
	L(x)=\sum_{|g|\le d}|g|x_g\lambda_g=\int_{\T} \sum_{|g|\le d}z^{|g|}x_g\lambda_gd\nu(z)=\int_{\T} \pi_z(x)d\nu(z).
	\end{equation}
	This, together with the triangular inequality, \eqref{eq:rotation free} and the definition of $\nu$, implies
	$$\|L(x)\|_p\le \int_{\T} \left\| \pi_z(x)\right\|_p d|\nu|(z)
	= \int_{\T} \| x\|_p d|\nu|(z)\le d\|x\|_p.$$
	This finishes the proof of \eqref{ineq:spectral gap low degree}.
	By \eqref{ineq:spectral gap low degree}, for any $m\ge 1$:
	$$\|L^m (x)\|_p\le d^m \|x\|_p.$$
	So we have for all $t\ge 0$ that
	$$\|P_{-t}(x)\|_p\le \sum_{m\ge 0} \frac{t^m \|L^m (x)\|_p}{m!}\le \sum_{m\ge 0} \frac{(td)^m}{m!}\|x\|_p=e^{td}\|x\|_p,$$
	and the proof of \eqref{ineq:heat smoothing low degree} is complete.
\end{proof}

\begin{rem}
We would like to thank the referee who pointed out that  \eqref{ineq:spectral gap low degree} follows from the vector-valued Bernstein's inequality. In fact, for holomorphic polynomial $x=\sum_{0\le k\le d}x_k$ with $L(x_k)=kx_k$, consider 
$$P(z):=\sum_{0\le k\le d}z^k x_k=\pi_z(x)$$ 
with 
$$P'(z):=\sum_{1\le k\le d}z^{k-1} k x_k=\frac{1}{z}\pi_z(L(x)),\qquad z\neq 0.$$
Note that  \eqref{eq:rotation free} for $1\le p\le \infty$ and $|z|=1$:
$$
\|P(z)\|_p\equiv\|x\|_p \qquad\textnormal{and}\qquad \|P'(z)\|_p\equiv\|L(x)\|_p,
$$
we have by vector-valued Bernstein's inequality that
$$\|L(x)\|_p=\sup_{|z|=1}\|P'(z)\|_p\le d\sup_{|z|=1}\|P(z)\|_p=d\|x\|_p.$$

\end{rem}

Proposition \ref{prop:moment comparison} follows from the same argument in \cite{EI20}. 

\begin{proof}[Proof of Proposition \ref{prop:moment comparison}]
Fix holomorphic $x\in\P^{\le d}$. Note that following the same argument in the proof of Theorem \ref{thm:heat smoothing free} (ii), one has
$$\|P_{-t}(x)\|_q\le e^{td}\|x\|_q, \qquad t>0,$$
where the definition of $P_{-t}$ is clear over $\P^{\le d}.$ 	Suppose that for such $x$, we have
\begin{equation}\label{ineq:hypercontractivity}
\|P_t(x)\|_q\le \|x\|_p,
\end{equation}
for some $(p,q,t)\in (1,\infty)\times (1,\infty)\times (0,\infty)$. Then we get the moment comparison inequalities:
\begin{equation}\label{ineq:moment comparison}
\|x\|_q=\|P_{-t} P_{t}(x)\|_q\le e^{td}\|P_t(x)\|_q\le e^{td}\|x\|_p.
\end{equation}
According to \cite[Theorem A ii)]{JPPR15}, the hypercontractivity \eqref{ineq:hypercontractivity} holds whenever $1< p\le q<\infty$ and $e^{t}\ge \frac{q-1}{p-1}$, which proves (i). By \cite[Theorem 10]{RX16} and its proof, the hypercontractivity \eqref{ineq:hypercontractivity} is valid for $p=2$, $q\ge 4-\epsilon_0\approx 3.82$ and $e^{2t}\ge q-1$. This shows (ii) and finishes the proof. 
\end{proof}

Now we turn to the proof of Theorem \ref{thm:free sharp order} which comes from the referee.

\begin{proof}[Proof of Theorem \ref{thm:free sharp order}]
	Let $m=(m_k)_{k\ge 0}\subset\com$ be a sequence that vanishes at $\infty$. By a result of Haagerup, Streenstrup and Szwarc \cite[Theorem 4.2, Lemma 1.10]{haagerup2010schur}, the radial Fourier multiplier $T_m$ on $\widehat{\F}_n$
	$$T_m: \lambda_g \mapsto m_{|g|} \lambda_g,$$
	is completely bounded if and only if the matrix $H_m=[m_{i+j}-m_{i+j+2}]_{i,j\ge 0}$ is of trace class. Moreover (actually the left hand side can be replaced by the cb norm), 
	\begin{equation*}
		\|T_m: \widehat{\F}_n\to \widehat{\F}_n\|\le \|H_m\|_1.
	\end{equation*}
	Now for each $d\ge 1$ and $r\in (0,1)$ we take the sequence $m=m(d,r)$ as follows. When $k\ge d$, we choose $m_k=r^k$, so that 
	$$m_k-m_{k+2}=r^k(1-r^2),\qquad k\ge d.$$
	When $k<d$, we choose $m_k$ such that 
	$$m_k-m_{k+2}=r^{2d-k}(1-r^2), \qquad 0\le k<d.$$
	Clearly, $\lim\limits_{k\to \infty}m_k=0$. To estimate $\|H_m\|_1$, we divide $H_m$ into three parts:
	$$H_m=p_d H_m+ p_d H_m (1-p_d)+(1-p_d)H_m=A+B+C,$$
	where $p_d:=\sum_{k=0}^{d}e_{i,i}$ is the projection over the first $d+1$ entries. 
	
	Note that for $\alpha=(\alpha_i)_{i\ge 0}$ and $\beta=(\beta_j)_{j\ge 0}$, the trace norm of $[\alpha_i \beta_j]_{i,j\ge 0}$ is $\|\alpha\|_2\|\beta\|_2$. So for $B=[r^{i+j}(1-r^2)]_{0\le i\le d, j>d}$ we have $\|B\|_1=r^{d+1}\sqrt{1-r^{2(d+1)}}\le r^{d+1}$, and for 
	$C=[r^{i+j}(1-r^2)]_{i>d,j\ge 0}$, we have $\|C\|_1=r^{d+1}$.
	
	By definition, $A=[a_{ij}]_{0\le i,j\le d}$ where 
	$$a_{ij}=\begin{cases}
		r^{2d-(i+j)}(1-r^2)&i+j\le d\\
		r^{i+j}(1-r^2)	& d< i+j\le 2d
	\end{cases},$$
	or equivalently, $a_{ij}=(1-r^2)r^d\cdot r^{|i+j-d|}$. Now for each $0\le j\le \lfloor d/2\rfloor$, we permute column $j$ and column $d-j$ to obtain a matrix $\widetilde{A}=AU=[\widetilde{a}_{ij}]_{0\le i,j\le d}$ for some unitary $U$, with $\widetilde{a}_{ij}=(1-r^2)r^d\cdot r^{|i-j|}$. Since the kernel $(i,j)\mapsto |i-j|$ is conditionally negative definite (see e.g. \cite[Example 3.5]{wells2012embeddings} or \cite{Haagerup79}), $\widetilde{A}=[\widetilde{a}_{ij}]$ is positive semi-definite by Schoenberg's theorem \cite{schoenberg1938metric}. So $\|A\|_1=\|\widetilde{A}\|_1=\textnormal{Tr}[\widetilde{A}]=(d+1)(1-r^2)r^d$.
	
	All combined, we get
	\begin{equation}\label{ineq:bd of multiplier}
		\|T_m: \widehat{\F}_n\to \widehat{\F}_n\|\le\|H_m\|_1\le \|A\|_1+\|B\|_1+\|C\|_1\le (d+1)(1-r^2)r^d+2r^{d+1}.
	\end{equation}
	As a consequence, for any $t> 0$ and for any $x\in \P^{\ge d}$, we have (choosing $r=e^{-t}$)
	\begin{equation}\label{ineq:norm estimate}
		\|P_t(x)\|=\|T_m(x)\|\le [(d+1)(1-e^{-2t})e^{-td}+2e^{-t(d+1)}]\|x\|.
	\end{equation}
A rough estimate 
$$(d+1)(1-e^{-2t})e^{-td}+2e^{-t(d+1)}\le (d+1)e^{-td}+2e^{-td}=(d+3)e^{-td},$$
gives \eqref{ineq:free hs sharp semigroup} for $p=\infty$. By \eqref{ineq:norm estimate},
	\begin{align*}
		\|x\|\le \int_{0}^{\infty}\|P_t L(x)\|dt\le& \int_{0}^{\infty}[(d+1)(1-e^{-2t})e^{-td}+2e^{-t(d+1)}]dt\| L(x)\|\\
		=&\left[(d+1)\left(\frac{1}{d}-\frac{1}{d+2}\right)+\frac{2}{d+1}\right]\|L(x)\|\\
		\le &\frac{4}{d}\|L(x)\|,
	\end{align*}
which proves  \eqref{ineq:free hs sharp generator} for $p=\infty$.
	Since \eqref{ineq:bd of multiplier} holds on the whole von Neumann algebra $\widehat{\F}_n$ and $T_m$ is self-adjoint, one obtains by duality and complex interpolation that for all $1\le p\le \infty$:
	\begin{equation*}
		\|T_m: L_p(\widehat{\F}_n)\to L_p(\widehat{\F}_n)\| \le (d+1)(1-r^2)r^d+2r^{d+1}.
	\end{equation*}
	Then following a similar argument we get  \eqref{ineq:free hs sharp semigroup} and \eqref{ineq:free hs sharp generator} for general $1\le p\le \infty$.
\end{proof}

\subsection{Proofs for $q$-Gaussian algebras and quantum tori}
The proofs of Theorems \ref{thm:heat smoothing q-gaussian} and \ref{thm:heat smoothing quantum tori} are essentially the same as that of Theorem \ref{thm:heat smoothing free}.

\begin{proof}[Proof of Theorems \ref{thm:heat smoothing q-gaussian} and \ref{thm:heat smoothing quantum tori}]
	To prove (i) in Theorems \ref{thm:heat smoothing q-gaussian} and \ref{thm:heat smoothing quantum tori}, note first that we still have an average identity similar to \eqref{eq:average-semigroup-free} by replacing \eqref{eq:defn of semigroup-free} \& \eqref{eq:rotation over analytic-free} with \eqref{eq:semigroup over analytic-q gaussian} \& \eqref{eq:rotation over analytic-q gaussian}, and \eqref{eq:defn of semigroup-quantum tori} \& \eqref{eq:rotation-analytic-quantum tori} respectively. Then one can finish the proof using the rotation invariance properties \eqref{eq:rotation invariance-q gaussian} and \eqref{eq:rotation quantum tori}, respectively, instead of \eqref{eq:rotation free}.
	
	The proof of (ii) can be adapted easily as we did for (i), by replacing \eqref{eq:defn of generator-free} \& \eqref{eq:rotation free} \& \eqref{eq:rotation over analytic-free} with \eqref{eq:generator over analytic-q gaussian} \& \eqref{eq:rotation invariance-q gaussian} \& \eqref{eq:rotation over analytic-q gaussian}, and \eqref{eq:defn of generator-quantum tori} \& \eqref{eq:rotation quantum tori} \& \eqref{eq:rotation-analytic-quantum tori}, respectively.
\end{proof}

The proof of Proposition \ref{prop:moment comparison q} is similar to that of Proposition \ref{prop:moment comparison}.

\begin{proof}
	The proof follows the same line as in the proof of Proposition \ref{prop:moment comparison}. The only difference is to use the holomorphic heat-smoothing estimates in Theorem \ref{thm:heat smoothing q-gaussian} and hypercontractivity results Theorems \ref{thm:hypercontracitivity for q} and \ref{thm:strong hypercontracitivity for q}.
\end{proof}

\subsection{Remarks}\label{subsec:ending remark}
The main results of sharp holomorphic heat-smoothing hold as long as one has the rotation invariance property \eqref{eq:rotation free} and nice holomorphic structure that gives the average identities \eqref{eq:average-semigroup-free} and \eqref{eq:average-generator-free}. Here we treated $q$-Gaussian algebras for $q\in (-1,1)$, but similar results hold for $q=\pm 1$ as well. When $q=1$, it corresponds to the classical Gaussian setting, so the work in \cite{EI20}.

 When $q=-1$, it gives the CAR algebra. By embedding $L_\infty(\{\pm 1\}^n,d\mu_n)$ into $2^n$-by-$2^n$ matrix algebra, Ben Efraim and Lust-Piquard \cite{LP08poincare} proved many inequalities, such as $L_p$-Poincar\'e, over the discrete hypercubes using rotation on matrix algebras. The rotation technique is similar to what we used here, so they also obtained inequalities for CAR algebras, in parallel to discrete hypercubes. Having  holomorphic heat-smoothing for CAR algebras, one may ask if we can obtain heat-smoothing for certain (``holomorphic'') functions on discrete hypercubes using the idea of Ben Efraim and Lust-Piquard. Unfortunately, although we have nice rotation operators in \cite[Section 2]{LP08poincare} and we may derive inequalities for the holomorphic variable $Z_j:=(Q_j+iP_j)/\sqrt{2}$, the operators $Z_j$'s, unlike $Q_j$'s, do not correspond to functions on the discrete hypercubes.  

\subsection*{Acknowledgement}  I am grateful to Alexandros Eskenazis and Paata Ivanisvili for helpful comments on an earlier version of the paper. I would like to thank the referee for valuable comments and remarks that improve the paper substantially, especially for Theorem \ref{thm:free sharp order}. Part of the work was finished during visits to LMB, Besan\c con, and IMPAN, Warsaw. I want to thank both places, and Professor Quanhua Xu and Adam Skalski for their warm hospitality. The research is supported by the Lise Meitner fellowship, Austrian Science Fund (FWF) M3337.

\newcommand{\etalchar}[1]{$^{#1}$}

\end{document}